\documentclass[11pt]{article}
\pdfoutput=1
\usepackage{amsfonts,amssymb,amsmath,amsthm,mathtools,indentfirst,
	enumerate,xcolor,hyperref,authblk,dsfont,tikz,caption}
\usepackage[a4paper,margin=1in]{geometry}
\usetikzlibrary{arrows.meta,spy,calc}
\usepackage[english]{babel}
\hypersetup{colorlinks,linkcolor={red!50!black},
	citecolor={blue!50!black}}

\makeatletter
\renewcommand{\fnum@figure}{Fig.~\thefigure}
\makeatother

\newcommand{\dd}{\mathrm d}

\newcommand{\abs}[1]{\left|#1\right|}

\newcommand{\hs}{\hspace{1pt}}
\newcommand{\R}{\mathbb{R}}
\newcommand{\X}{\mathbb{X}}
\newcommand{\N}{\mathbb{N}}
\newcommand{\E}{\mathbb{E}}
\renewcommand{\P}[1]{\mathbb{P}{\left\{#1\right\}}}

\newcommand{\1}[1][]{
	\ifx&#1&\mathds{1}
	\else\mathds{1}\left\{#1\right\}\fi}

\newcommand{\dto}{\overset{d}{\longrightarrow}}
\newcommand{\deq}{\overset{d}{=}}

\newcommand{\Pois}[1]{\mathsf{Pois}(#1)}

\newcommand{\Cc}{\mathcal{C}_{\eps,\delta}}
\newcommand{\Pc}{\mathcal{P}}
\newcommand{\Bc}{\mathcal{B}}
\newcommand{\Ic}{\mathcal{I}}
\newcommand{\Sym}[1]{\mathfrak{S}_{#1}}
\newcommand{\eps}{\varepsilon}

\DeclareMathOperator{\supp}{supp}
\DeclareMathOperator*{\argmin}{arg\,min}
\DeclarePairedDelimiter\floor{\lfloor}{\rfloor}

\newtheorem{theorem}{Theorem}[section]
\newtheorem{lemma}{Lemma}[section]
\newtheorem{proposition}{Proposition}[section]
\newtheorem{corollary}{Corollary}[section]
\theoremstyle{definition}

\theoremstyle{remark}
\newtheorem{remark}{Remark}[section]

\numberwithin{equation}{section}

\title{\bfseries Scaling limit for small blocks\\in the Chinese restaurant process}
\author{Oleksii Galganov and Andrii Ilienko}
\date{}

\begin{document}
	\maketitle
	\begingroup
	\renewcommand\thefootnote{}
	\footnotetext{The second-named author was supported by the Swiss National Science Foundation, grant no.~229505.}
	\addtocounter{footnote}{-2}
	\endgroup
	\vspace{-25pt}
	
	\subsection*{Abstract}
			The Chinese restaurant process is a basic sequential construction of consistent random partitions. 
			We consider random point measures describing the composition of small blocks in such partitions and show that their 
			scaling limit is given by the projective limit of certain inhomogeneous Poisson measures on cones of increasing dimension. 
			This result makes it possible to derive classical and functional limit theorems in the Skorokhod topology 
			for various characteristics of the Chinese restaurant process.
	
	\section{Introduction}
	The Chinese restaurant process (CRP), introduced by Dubins and Pitman in the early 1980s, is a preferential attachment algorithm for constructing consistent random 
	partitions (and associated permutations) on a single probability space; see Chapter 3 of \cite{P06}. The construction proceeds as follows. Fix a parameter $\theta>0$, 
	and initialize the process by placing the element $1$ in the first block. The element $2$ either starts a new block with probability $\frac{\theta}{\theta+1}$ or joins 
	the first block with probability $\frac{1}{\theta+1}$. After $n$ elements have been assigned to blocks, the $n+1$-th element either starts a new block with probability 
	$\frac{\theta}{\theta+n}$ or joins an existing block $B_k$ with probability $\frac{\abs{B_k}}{\theta+n}$, where $\abs{B_k}$ denotes the cardinality of block $B_k$. 
	This procedure generates a consistent sequence of random partitions of sets $[n]=\{1,\dots,n\}$ for $n\in \N$. 
	The process can also be adapted to generate a sequence of random permutations by interpreting each block of the partition as a cycle of the permutation.
	In this interpretation, each newly arriving element is inserted to the right of a uniformly chosen element within the selected cycle or starts a new cycle.
		
	The CRP and related processes find numerous applications, e.g., in infinite hidden Markov models, topic models, and network analysis, 
	as it enables flexible expansion of clusters or states with a continuous influx of new data \cite{BGR01,TJBB06,FSJW11}.
	A whole class of Markov algorithms generating random partitions and permutations, related to the CRP, was introduced in a recent study \cite{S24}.
	
	It is well known that, for each $n$, the random permutations constructed in this way follow the Ewens distribution:
	\begin{equation*}
		\P{\sigma_n = \pi} = \frac{\theta^{C(\pi)}} {\theta^{\overline n}}, 
		\qquad \pi\in\Sym n,
	\end{equation*}
	where $C(\pi)$ is the total number of cycles in $\pi$, 
	$\theta^{\overline n}=\theta(\theta+1)\ldots(\theta+n-1)$ stands for the rising factorial, 
	and $\Sym n$ denotes the symmetric group of order $n$. 
	This distribution and its counterpart for random partitions are based on the celebrated Ewens sampling formula, 
	which originated in population genetics and subsequently found wide applications across diverse fields; see \cite{Cr16a,Cr16b} and numerous references therein. 
	A comprehensive account of the properties of such permutations is provided in \cite{ABT03}. In particular, Theorem 5.1 demonstrates that
	\begin{equation}
		\label{eq:ascc}
		\bigl(C_k(\sigma_n), k\in\N\bigr) \dto \bigl(Z_k, k\in\N\bigr), \qquad n\to\infty,
	\end{equation}
	where $C_k$ denotes the number of cycles of size $k$, and $Z_k$ are independent random variables following
	$\mathsf{Pois}\bigl(\frac{\theta}{k}\bigr)$. Note that, for $\theta=1$, Ewens distribution becomes uniform on $\Sym n$.
	
	The CRP also admits an alternative, seemingly unrelated construction in terms of the Kingman paintbox process. In this setting, random partitions and permutations arise from an occupancy scheme with random probabilities generated by stick-breaking of the unit interval with $\mathsf{Beta}(\theta,1)$ factors \cite{DJ91,G04,GIM12}. This connects the present setting to the recent work \cite{DGM24} and, in part, to the second author’s paper \cite{I19}. The former studies an infinite occupancy scheme with fixed (and, in its final section, certain random) probabilities and proves that the suitably rescaled times at which a box receives its $r$th ball converge to a Poisson point process. The latter obtains similar results for an occupancy scheme with finitely many equiprobable boxes. However, as we emphasize below, these and related papers track only the counts of balls in the boxes rather than composition of the boxes.
	
	Some dynamic properties of random partitions arising in the CRP were studied in~\cite{GS23}. It was shown that this process exhibits rather irregular behavior in discrete time but can be regularized by embedding it into continuous time. By applying results from queueing theory, this approach yields a functional limit theorem for small block counts, with the limit being a certain time-changed stationary continuous-time Markov chain.
	Block counts for other variants of the CRP have also recently attracted attention in the literature. In~\cite{GW24}, limit theorems were established for linear combinations of block counts in the CRP with $(\alpha,\theta)$-seating for $\alpha>0$; the classical case corresponds to $\alpha=0$. Maximum block counts for the disordered CRP were investigated in~\cite{BMMU24}.
	
	In all the above papers, the study of block dynamics was limited exclusively to block counts viewed as (random) numerical functions of time. However, the very definition of the CRP suggests a more challenging problem: describing the dynamics of not merely the block counts but the entire composition of these blocks. Clearly, such a composition cannot be captured by scalar- or vector-valued random processes.
	In this paper, we study the limiting composition of small blocks in the CRP through the convergence of the corresponding random point measures. A related approach 
	was used in our previous paper \cite{GI24}, where the focus was on point measures describing the composition of short cycles in random permutations of fixed length, and 
	their vague convergence, as the length grows to infinity, to a homogeneous Poisson measure on a specially constructed metric space. However, in the static setting of 
	that study, the appearance of homogeneous Poisson measures in the limit was not particularly surprising due to the invariance of such random permutations under 
	relabelling. The situation here is quite different: when studying dynamic characteristics, inhomogeneity arises naturally, as the composition of blocks at future times 
	depends on their composition at earlier times. Surprisingly, with a proper construction of pre-limit measures, this inhomogeneity takes a remarkably simple form, 
	enabling the derivation of classical and functional limit theorems for a wide range of dynamic characteristics of the CRP far beyond block counts. A number of such theorems are given in Section \ref{sec4}.
	
	\section{Preliminaries and main result}\label{sec2}
	
	Let $\Pc_n$, $n\in\N$, denote the partition of $[n]$ formed at the $n$th step of the CRP. 
	For convenience, we list the elements within each block of $\Pc_n$ in ascending order and sort the blocks themselves by their smallest elements.
	Thus, $\Pc_1=\{\{1\}\}$, $\Pc_2=\{\{1\},\{2\}\}$ with probability $\frac{\theta}{\theta+1}$ and $\{\{1,2\}\}$
	with probability $\frac{1}{\theta+1}$, and so on. The growth of each block is described by the following scheme:
	\begin{equation}
		\label{eq:sch}
		\{k_1\} \longrightarrow \{k_1,k_2\} \longrightarrow \cdots
		\longrightarrow \{k_1,\ldots,k_N\} \longrightarrow
		\{k_1,\ldots,k_N,k_{N+1}\} \longrightarrow \cdots,
	\end{equation}
	where $1\le k_1<k_2<\ldots$, and the block $\{k_1\}$ appears at step $k_1$ with probability $\frac{\theta}{\theta+k_1-1}$,
	$k_2$ joins it at step $k_2$ with probability $\frac{1}{\theta+k_2-1}$, $k_3$ joins at step $k_3$ with probability
	$\frac{2}{\theta+k_3-1}$, and so on. Note that, by the Borel-Cantelli lemma, each block will a.s.~grow infinitely.
	Denote by $A(k_1,\ldots,k_N,k_{N+1})$ the random event indicating the existence of a block in the CRP that,
	up to the time $k_{N+1}$, evolves exactly as described in \eqref{eq:sch}; that is,
	\begin{equation}
		\label{eq:A}
		A(k_1,\ldots,k_N,k_{N+1}) = \bigl\{\{k_1,\ldots,k_N,k_{N+1}\}
		\in \Pc_{k_{N+1}}\bigr\}.
	\end{equation}
	
	Now fix $N\in\N$ and consider the structure and dynamics of all blocks up to the times when they reach size $N+1$.
	It is clear from \eqref{eq:sch} that it suffices to specify the elements $k_1,\ldots,k_N$ with which such a block
	eventually (or, more precisely, exactly at step $k_N$) reaches size $N$, and the time $m=k_{N+1}$
	when it gains one more element and thus stops being tracked.
	Hence, such structure and dynamics are uniquely determined by the infinite collection of events
	\begin{equation*}
		\{A(k_1,\ldots,k_N,m),1\le k_1<\ldots<k_N<m\},
	\end{equation*}
	or, equivalently, by the random point measure
	\begin{equation}
		\label{eq:Xi1}
		\Xi_1^{(N)} =
		\sum_{1\le k_1 < \ldots < k_N < m} \delta_{(k_1,\ldots,k_N,m)}
		\1_{A(k_1,\ldots,k_N,m)}.
	\end{equation}
	
	Introduce a cone
	\begin{equation}
		\label{eq:XN}
		\X_N = \left\{(x_1,\ldots,x_N,y) \in (0,+\infty)^{N+1}: x_1\le\ldots\le x_N\le y\right\}.
	\end{equation}
	Considering it as a measurable space, we equip it with the Borel $\sigma$-algebra $\Bc(\X_N)$, the measure $\mu^{(N)}$ defined by
	\begin{equation}
		\label{eq:mu}
		\dd\mu^{(N)} = \theta\,\frac{N!}{y^{N+1}}\,\dd x_1\ldots\dd x_N\,\dd y,
	\end{equation}
	and the localizing ring of bounded sets
	\begin{equation}
		\label{eq:loc}
		\mathcal{X}_N = \bigl\{B\in\Bc(\X_N):
		\text{$B$ is bounded away from zero and }
		\mu^{(N)}(B)<\infty\bigr\};
	\end{equation}
	for details on the latter, see, e.g., \cite[p.~19]{K17}. By scaling \eqref{eq:Xi1}, 
	define a sequence of random point measures $\Xi_n^{(N)}$, $n\in\N$, on $\left(\X_N,\Bc(\X_N)\right)$ as
	\begin{equation*}
		\Xi_n^{(N)} = 
		\sum_{1\le k_1 < \ldots < k_N < m} 
		\delta_{\left(\frac{k_1}{n},\ldots,\frac{k_N}{n},\frac{m}{n}\right)}
		\1_{A(k_1,\ldots,k_N,m)}.
	\end{equation*}
	
	The following theorem provides a complete description of the asymptotic structure and dynamics for blocks of size up to $N$.
	Recall that the vague topology on the space of locally finite 
	measures is generated by the integration maps 
	$\nu\mapsto\int_\X f\,\dd\nu$ for all continuous functions 
	$f$ with bounded support; see, e.g., Section 3.4 in \cite{R87} 
	or Chapter 4 in \cite{K17} for a general exposition.
	
	\begin{theorem}
		\label{TH:MAIN_TH}
		$\Xi_n^{(N)}$ vaguely converge in distribution as $n\to\infty$ to
		the Poisson random measure $\Xi^{(N)}$ on $\left(\X_N,\Bc(\X_N)\right)$ with
		intensity measure $\mu^{(N)}$.
	\end{theorem}
	
	The measures $\mu^{(N)}$ for different $N$ are consistent in the following sense. Temporarily denoting $y$ in \eqref{eq:XN} by $x_{N+1}$, for $M>N$ and $B_N\in\Bc(\X_N)$, define
	\begin{equation*}
		B_{N\uparrow M} = \{(x_1,\ldots,x_{M+1}):
		(x_1,\ldots,x_{N+1})\in B_N,x_{N+1} \le \ldots \le x_{M+1}\}
		\in \Bc(\X_M).
	\end{equation*}
	Then
	\begin{align*}
		\mu^{(M)}(B_{N\uparrow M}) &= \int_{B_N}\dd x_1\ldots\dd x_{N+1}\int_{x_{N+1}}^{+\infty}\dd x_{N+2}\ldots
		\int_{x_{M-1}}^{+\infty}\dd x_M \int_{x_M}^{+\infty}\theta\,
		\frac{M!}{x_{M+1}^{M+1}}\,\dd x_{M+1}
		\\ &= \int_{B_N}\theta\,\frac{N!}{x_{N+1}^{N+1}}\,\dd x_1\ldots
		\dd x_{N+1}=\mu^{(N)}(B_N).
	\end{align*}

	Hence, the distributions of $\Xi^{(N)}$, $N\in\N$, are also consistent. Thus, we can define their projective limit $\Xi^{(\infty)}$.
	Similarly, for any $n\in\N$, we can define the projective limit $\Xi_n^{(\infty)}$ of $\Xi_n^{(N)}$.
	There is, however, a principal difference between these two projective limits. While the latter can be interpreted as
	the distribution of a random point measure on the infinite-dimensional cone of \emph{all} non-decreasing sequences
	with positive terms, representing the composition of \emph{all} blocks, the former is no longer the distribution of
	a Poisson measure on such a cone, since the right-hand side of \eqref{eq:mu} diverges as $N\to\infty$.
	Nevertheless, Theorem \ref{TH:MAIN_TH} can be equivalently stated as a result on the vague convergence of $\Xi_n^{(\infty)}$ to $\Xi^{(\infty)}$.
	
	\begin{remark}
		\label{rem:si}
		The limiting processes $\Xi^{(N)}$ are scale-invariant: $\Xi^{(N)}(B)\deq\Xi^{(N)}(cB)$ for any $c>0$ and $B\in\Bc(\X_N)$. This follows from the scale invariance of the intensity measure $\mu^{(N)}$.
	\end{remark}
	
	\section{Proof of Theorem \ref{TH:MAIN_TH}}\label{sec3}
	
	We will precede the proof with two auxiliary lemmas.
	As before, we write $[r]$ for $\{1,\ldots,r\}$ and $|B|$ for the cardinality of $B$.
	
	\begin{lemma}
		\label{lem:1}
		Let $r,N \in \N$, and $\bigl(k_1^{(s)},\dots,k_N^{(s)}, m^{(s)}\bigr)$, $s\in [r]$, be disjoint tuples of increasing positive integers. Denote
		\begin{equation}
			\label{eq:ls}
			l^{(s)} = \bigl|\bigl\{k_p^{(s')},p\in[N],s'\in[r]:
			k_p^{(s')}<m^{(s)}, m^{(s')} > m^{(s)}\bigr\}\bigr|.
		\end{equation}
		Then
		\begin{equation}
			\label{eq:joint_prob}
			\mathbb{P}\biggl\{\bigcap_{s=1}^r A\bigl(k_1^{(s)}, \dots, k_N^{(s)}, m^{(s)}\bigr)\biggr\} = 
			\theta^r\prod_{s=1}^r \frac{N!}{\left(\theta + m^{(s)} - l^{(s)} - 1\right)^{\underline{N+1}}},
		\end{equation}
		where the events $A$ are defined in \eqref{eq:A}, and $x^{\underline n} = x (x-1) \ldots (x-n+1)$ is the falling factorial.
	\end{lemma}
	\begin{proof}
		Let $a_j$, $j\in[r(N+1)]$, denote the collection of $k_1^{(s)}, \dots, k_N^{(s)}, m^{(s)}$ for all $s\in[r]$, sorted in ascending order. For each $j$, define 
		\begin{equation}
			\label{eq:Kj}
			K_j = \begin{cases}
				\bigl\{k_1^{(s)}\bigr\}, &\text{if } a_j = k_1^{(s)}\text{\hspace{6.3pt}for some }s\in[r],\\
				\bigl\{k_1^{(s)},\ldots,k_{p-1}^{(s)}\bigr\},
				&\text{if } a_j = k_p^{(s)}\text{\hspace{6.3pt}for some }s\in[r]\text{ and }p\in[N]\setminus\{1\},\\
				\bigl\{k_1^{(s)}, \dots, k_N^{(s)}\bigr\},
				&\text{if } a_j = m^{(s)}\text{ for some }s\in[r].
			\end{cases}
		\end{equation}
		Additionally, let
		\begin{equation}
			\label{eq:Lj}
			L_j = \bigl\{k_p^{(s')},p\in[N],s'\in[r]: k_p^{(s')} \le a_j, m^{(s')}>a_j\bigr\}.
		\end{equation}
		In particular, it follows from \eqref{eq:ls} and \eqref{eq:Lj} that
		\begin{equation}
			\label{eq:ml}
			|L_j| = l^{(s)} \quad\text{if}\quad a_j=m^{(s)}.
		\end{equation}
		
		Let $\Ic_n$, $n\in\N$, be independent random variables distributed as
		\begin{equation}
			\label{eq:Idist}
			\P{\Ic_n = k} = \begin{cases}
				\frac{\theta}{\theta + n - 1}, & k = n, \\
				\frac{1}{\theta + n - 1}, & k \in [n-1].
			\end{cases}
		\end{equation}
		It follows from the construction of the CRP that element $n$ at time $n$ starts a new block if $\Ic_n=n$,
		and joins an already existing block containing element $\Ic_n$ otherwise.
		Then, the left-hand side of \eqref{eq:joint_prob} can be written as
		\begin{equation}
			\begin{aligned}
				\label{eq:joint_prob_expanded}
				\P{\Ic_{a_1}\in K_1}&\cdot\!\!\prod_{i={a_1+1}}^{a_2 - 1}\P{\Ic_i \notin L_1}\cdot\P{\Ic_{a_2}\in K_2}\cdot\ldots\\&\times\!\! \prod_{i=a_{r(N+1)-1}+1}^{a_{r(N+1)}-1}\!\!\P{\Ic_i
					\notin L_{r(N+1)-1}}\cdot\P{\Ic_{r(N+1)} \in K_{r(N+1)}}.
			\end{aligned}
		\end{equation}
		This is best explained with a specific example. Let $r=2$, $N=3$,
		\begin{equation}
			\label{eq:ex}
			\bigl(k_1^{(1)},k_2^{(1)},k_3^{(1)},m^{(1)}\bigr) = (3,7,11,19),\qquad
			\bigl(k_1^{(2)},k_2^{(2)},k_3^{(2)},m^{(2)}\bigr) = (6,12,21,24).
		\end{equation}
		Hence, we have $(a_1,\ldots,a_8)=(3,6,7,11,12,19,21,24)$,
		\begin{equation*}
			\begin{aligned}
				K_1 &= \{3\}, & K_2 &= \{6\}, & K_3 &= \{3\}, & K_4 &= \{3,7\}, \\
				K_5 &= \{6\}, & K_6 &= \{3,7,11\}, & K_7 &= \{6,12\}, & K_8 &= \{6,12,21\}
			\end{aligned}
		\end{equation*}
		by \eqref{eq:Kj}, and
		\begin{equation*}
			\begin{aligned}
				L_1 &= \{3\}, & L_2 &= \{3,6\}, & L_3 &= \{3,6,7\}, & L_4 &= \{3,6,7,11\}, \\
				L_5 &= \{3,6,7,11,12\}, & L_6 &= \{6,12\}, & L_7 &= \{6,12,21\}, & L_8 &= \varnothing
			\end{aligned}
		\end{equation*}
		by \eqref{eq:Lj}.
		Thus, to get the blocks \eqref{eq:ex} in $\Pc_{24}$, it is necessary to fall into $\{3\} = K_1$ at step $3$,
		avoid $\{3\} = L_1$ at steps $4$ and $5$, fall into $\{6\} = K_2$ at step $6$, avoid $\{3,6\} = L_2$
		strictly between steps $6$ and $7$ (there are no such steps), fall into $\{3\} = K_3$ at step $7$,
		avoid $\{3,6,7\}=L_3$ at steps $8$ to $10$, and so on.    
		
		It is straightforward from \eqref{eq:Idist} and the definitions of $K_j$, $L_j$, and $\Ic_{a_j}$ that
		\begin{gather}
			\label{eq:lemma_1_prod_1}
			\prod_{j=1}^{r(N+1)}\P{\Ic_{a_j} \in K_j}= 
			(\theta \cdot N!)^r \prod_{j=1}^{r(N+1)} \frac{1}{\theta + a_j - 1}, \\
			\prod_{i={a_j + 1}}^{a_{j+1} - 1} \P{\Ic_i \notin L_j}=
			\prod_{i={a_j + 1}}^{a_{j+1}-1} \left(
			1 - \frac{|L_j|}{\theta + i - 1}\right) = 
			\frac{(\theta + a_j - 1)^{\underline{|L_j|}}}{(\theta + a_{j+1} - 2)^{\underline{|L_j|}}}
		\end{gather}
		as $j\in[r(N+1)-1]$.

		By \eqref{eq:Lj}, $L_{r(N+1)} = \varnothing$. Setting additionally $L_0 = \varnothing$, we get
		\begin{equation}
			\prod_{j=1}^{r(N+1)-1} \prod_{i={a_j + 1}}^{a_{j+1} - 1} \P{\Ic_i \notin L_j} = 
			\!\!\prod_{j=1}^{r(N+1)} \frac{(\theta + a_j - 1)^{\underline{|L_j|}}}{(\theta + a_j - 2)^{\underline{|L_{j-1}|}}}
		\end{equation}
		by means of an index shift.        
		From the definition \eqref{eq:Lj}, it is easy to see that $|L_j|-|L_{j-1}|$ equals $1$ if $a_j=k_p^{(s)}$ and $-N$ if $a_j=m^{(s)}$ for some $s$ and $p$.
		It implies that
		\begin{equation}
			\label{eq:ff}
			\frac{(\theta + a_j - 1)^{\underline{|L_j|}}}{(\theta + a_j - 2)^{\underline{|L_{j-1}|}}} = 
			\begin{cases}
				\frac{\theta + a_j - 1}{\left(\theta + a_j - |L_j| - 1\right)^{\underline{N+1}}}, & a_j \in \left\{m^{(1)}, \dots, m^{(r)}\right\}, \\
				\theta + a_j - 1, & \text{otherwise}.
			\end{cases}
		\end{equation}        
		
		Combining all factors in \eqref{eq:joint_prob_expanded} and taking into account \eqref{eq:lemma_1_prod_1}\hs--\hs\eqref{eq:ff}, we obtain
		\begin{equation*}
			\mathbb{P}\biggl\{\bigcap_{s=1}^r A\bigl(k_1^{(s)}, \dots, k_N^{(s)}, m^{(s)}\bigr)\biggr\} = 
			\theta^r \!\!\prod_{a_j\in\{m^{(1)}, \dots, m^{(r)}\}}\frac{N!}{\left(\theta + a_j - |L_j| - 1\right)^{\underline{N+1}}}.
		\end{equation*}
		In view of \eqref{eq:ml}, this coincides with \eqref{eq:joint_prob}.
	\end{proof}
	
	\begin{lemma}
		\label{lem:2}
		Under the conditions of Lemma \ref{lem:1}, let $U\in\mathcal X_N$ be a finite union of closed convex sets. Then, as $n\to\infty$,
		\begin{equation}
			\label{eq:joint_prob_sum}
			\begin{aligned}
				\hspace{-7pt}\sum_{\substack{1\le k_1^{(s)}<\ldots<k_N^{(s)}<m^{(s)}\\\forall s\in [r]}}
				\hspace{-10pt}\mathbb{P}\biggl\{\bigcap_{s=1}^r A{\bigl(k_1^{(s)}, \ldots, k_N^{(s)}, m^{(s)}\bigr)}\biggr\}\cdot
				\mathds 1{\biggl\{\forall s\in [r]: \Bigl(\frac{k_1^{(s)}}{n},
					\ldots,\frac{k_N^{(s)}}{n}, \frac{m^{(s)}}{n}\Bigr)\in U\biggr\}}\\ 
				\to \left(\mu^{(N)}(U)\right)^r\qquad
				\text{as }n\to\infty,
			\end{aligned}
		\end{equation}
		where $\mu^{(N)}$ is given by \eqref{eq:mu}.
	\end{lemma}
	\begin{proof}
		It follows from \eqref{eq:ls} that $0\le l^{(s)}\le rN$ for any $s\in[r]$. Hence, by Lemma \ref{lem:1} and the definition of $x^{\underline n}$, we have
		\begin{equation*}
			\theta^r\prod_{s=1}^r \frac{N!}{\left(\theta + m^{(s)}\right)^{N+1}} \le 
			\mathbb{P}\biggl\{\bigcap_{s=1}^r A\bigl(k_1^{(s)},\dots, k_N^{(s)}, m^{(s)}\bigr)\biggr\} \le 
			\theta^r\prod_{s=1}^r\frac{N!}{\left(\theta+m^{(s)}-(r+1)N- 1\right)^{N+1}}.
		\end{equation*}
		Since $\theta$, $r$, $N$ are fixed, it is easy to see that, for any $\eps>0$, there exists $M_\eps\in\N$ such that
		\begin{equation}
			\label{eq:bounds}
			(1-\eps)\cdot\theta^r\prod_{s=1}^r\frac{N!}{(m^{(s)})^{N+1}}\le 
			\mathbb{P}\biggl\{\bigcap_{s=1}^r A\bigl(k_1^{(s)},\dots, k_N^{(s)},\, m^{(s)}\bigr)\biggr\} \le 
			(1+\eps)\cdot\theta^r\prod_{s=1}^r\frac{N!}{(m^{(s)})^{N+1}},
		\end{equation}
		whenever $m^{(1)},\ldots,m^{(r)}\ge M_\eps$.
		Since $U$ is bounded away from zero, the $(N+1)$-th components of all elements in $U$ exceed some $y_U>0$.
		This means that the sum in \eqref{eq:joint_prob_sum} can only include tuples with $m^{(1)},\ldots,m^{(r)}>ny_U$.
		Thus, for any $\eps$, both bounds in \eqref{eq:bounds} hold for sufficiently large $n$.
		Hence, this sum is sandwiched between
		\begin{equation*}
			(1\pm\eps)\cdot\frac{1}{n^{N+1}} \sum_{\substack{1\le k_1^{(s)}<\ldots<k_N^{(s)}<m^{(s)}\\\forall s\in [r]}}
			\prod_{s=1}^r\frac{\theta\hs N!}{(m^{(s)}/n)^{N+1}}\,
			\1{\Bigl\{\forall s\in [r]: \Bigl(\frac{k_1^{(s)}}{n},\ldots,\frac{k_N^{(s)}}{n}, \frac{m^{(s)}}{n}\Bigr)\in U\Bigr\}}.
		\end{equation*}
		
		If $U$ is bounded in the Euclidean metric, then, by letting first $n\to\infty$ and then $\eps\to0$,
		we obtain the convergence of these integral sums to        
		\begin{equation*}
			\int_{U^r} \prod_{s=1}^r\biggl(\frac{\theta\hs N!}{(y^{(s)})^{N+1}}\biggr)
			\prod_{s=1}^r\,\dd x_1^{(s)}\ldots\,\dd x_N^{(s)}\,\dd y^{(s)} = 
			\biggl(\int_U\frac{\theta\hs N!}{y^{N+1}}\,\dd x_1\ldots \dd x_N\,\dd y\biggr)^r = 
			\bigl(\mu^{(N)}(U)\bigr)^r.
		\end{equation*}
		If $U$ is unbounded, then by the definition of $\X_N$, it is unbounded from above in its last component.
		Then, for a large $b$, $U$ can be divided by the hyperplane $\{y=b\}$ into the lower and upper parts $U_b^-$ and $U_b^+$.
		The set $U_b^-$ is bounded away from infinity, and the previous argument applies, while the pre-limit sums for $U_b^+$
		are uniformly small as $b\to\infty$ by \eqref{eq:bounds}. Alternatively, one can appeal to the fact that the decreasing
		function $y^{-(N+1)}$ is directly Riemann integrable, and hence the integral sums converge to the integral over the entire unbounded domain $U$.    
	\end{proof}
	
	\begin{proof}[Proof of Theorem \ref{TH:MAIN_TH}]
		Since any open subset of $\X_N$ is a countable union of closed convex sets from $\mathcal X_N$,
		and any set from $\mathcal X_N$ can be covered by finitely many such sets, the class of all sets $U$ from
		Lemma \ref{lem:2} forms a dissecting ring in the sense of \cite[p.~24]{K17}.
		Hence, by the well-known sufficient conditions for distributional vague convergence
		(see, e.g., Theorem 4.18 in the same source), it suffices to show that, for any such $U$,
		\begin{enumerate}[(i)]
			\item $\lim_{n\to\infty}\E\hs\Xi_n^{(N)}(U) = \E\hs\Xi^{(N)}(U)$,
			\item $\lim_{n\to\infty}\mathbb{P}\bigl\{\Xi_n^{(N)}(U)=0\bigr\} = \mathbb{P}\bigl\{\Xi^{(N)}(U)=0\bigr\}$,
		\end{enumerate}
		where $\E\hs\Xi^{(N)}(U) = \mu^{(N)}(U)$ and $\mathbb{P}\bigl\{\Xi^{(N)}(U)=0\bigr\} = \exp\bigl\{-\mu^{(N)}(U)\bigr\}$        
		by the definition of $\Xi^{(N)}$.
		
		(i) is nothing but the statement of Lemma \ref{lem:2} for $r=1$. To prove (ii), we note that        
		\begin{equation*}
			\mathbb{P}\bigl\{\Xi_n^{(N)}(U)=0\bigr\} = 
			1 - \mathbb{P}\biggl\{\bigcup_{(k_1/n,\ldots,k_N/n,m/n)\in U}\!\!\!A(k_1,\dots,k_N,m)\biggr\}.
		\end{equation*}
		Hence, by the Bonferroni's inequality, for any $R\in\N$, 
		\begin{equation*}
			\mathbb{P}\bigl\{\Xi_n^{(N)}(U)=0\bigr\}\le \sum_{r=0}^{2R}(-1)^r\!\!
			\sideset{}{^*}\sum_{\substack{
					\left(k_1^{(1)}/n,\ldots,k_N^{(1)}/n, m^{(1)}/n\right)\in U,\\[2pt]
					\cdot\;\;\cdot\;\;\cdot\\[2pt]
					\left(k_1^{(r)}/n,\ldots,k_N^{(r)}/n, m^{(r)}/n\right)\in U
			}}\!\!
			\mathbb{P}\biggl\{\bigcap_{s=1}^r A{\bigl(k_1^{(s)}, \ldots,k_N^{(s)},m^{(s)}\bigr)}\biggr\},
		\end{equation*}
		with a similar lower bound involving the sum $\sum_{r=0}^{2R-1}$. Here $\sum^*$ indicates that the inner sum
		is taken over all unordered sets of disjoint tuples. Thus, the inner sum is $r!$ times smaller than the sum
		in \eqref{eq:joint_prob_sum}. Therefore, it follows from Lemma \ref{lem:2} that
		\begin{align*}
			\sum_{r=0}^{2R-1}\frac{(-1)^r}{r!} \bigl(\mu^{(N)}(U)\bigr)^r&\le
			\liminf_{n\to\infty}\mathbb{P}\bigl\{\Xi_n^{(N)}(U)=0
			\bigr\}\\
			&\le\limsup_{n\to\infty}\mathbb{P}\bigl\{\Xi_n^{(N)}(U)=0\bigr\}            
			\le\sum_{r=0}^{2R}\frac{(-1)^r}{r!} \bigl(\mu^{(N)}(U)\bigr)^r
		\end{align*}
		for any $R\in\N$. Letting $R \to \infty$ yields
		\begin{equation*}
			\lim_{n\to\infty}\mathbb{P}\bigl\{\Xi_n^{(N)}(U)=0\bigr\} = \exp\bigl\{-\mu^{(N)}(U)\bigr\},
		\end{equation*}
		which proves (ii) and hence the theorem.
	\end{proof}
	
	\section{Limit theorems for characteristics of the CRP}\label{sec4}

	Theorem \ref{TH:MAIN_TH}, combined with the continuous mapping theorem, allows us to derive limit results
	for a variety of CRP characteristics with limits given in an explicit form. Note that, in the case of
	characteristics related solely to block counts, rather than to the specific composition of blocks,
	such limits can also be described implicitly by means of Theorem 6 in \cite{GS23} as corresponding functionals
	of certain time-changed stationary continuous-time Markov chains.
	
	\subsection{Limiting distributions of block counts}
	We begin with a result clarifying the asymptotic relationship between the block counts at times $n$ and $\floor{\alpha n}$, $\alpha>1$.
	
	\begin{proposition}
		\label{prop:counts}
		Fix $N\in\N$ and let $C_k(\Pc_n)$, $k\in [N]$, denote the number of blocks of size $k$ in $\Pc_n$. Then, for any $\alpha>1$, we have
		\begin{equation}
			\label{eq:nan-conv}
			\begin{aligned}
				\bigl(C_1(\Pc_n), C_2(\Pc_n), \ldots, C_N(\Pc_n);\;
				& C_1(\Pc_{\floor{\alpha n}}),
				C_2(\Pc_{\floor{\alpha n}}), \ldots,
				C_N(\Pc_{\floor{\alpha n}}) \bigr) \\ 
				\dto\Bigl(&
				\sum_{j=1}^N X_{1j} + X_{1,>N},
				\sum_{j=2}^N X_{2j} + X_{2,>N},\ldots,
				X_{NN} + X_{N,>N}; \\&
				X_{01} + X_{11},X_{02} + X_{12} + X_{22},\ldots,
				\sum_{i=0}^N X_{iN}\Bigr), \qquad n\to\infty,
			\end{aligned}
		\end{equation}
		where all $X_{ij}$, $1 \le j \le N$, $0 \le i \le j$, and $X_{i, >N}$, $1 \le i \le N$, are independent and Poisson distributed with means
		\begin{equation}
			\label{eq:lambda}
			\lambda_{ij} = \frac{\theta}{j}\hs\binom{j}{i} \bigl(\alpha^{-1}\bigr)^i
			\bigl(1-\alpha^{-1}\bigr)^{j-i} \quad \text{and} \quad
			\lambda_{i,>N} = \frac{\theta}{i}\hs I_{1-\alpha^{-1}}(N-i+1,i).
		\end{equation}
		Here
		\begin{equation*}
			I_x(a,b) = \frac{B_x(a,b)}{B(a,b)} =
			\frac{\int_0^xt^{a-1}(1-t)^{b-1}\,\dd t}
			{\int_0^1t^{a-1}(1-t)^{b-1}\,\dd t}, \qquad x\in[0,1],\, a,b>0,
		\end{equation*}
		stands for the normalized incomplete beta function.
	\end{proposition}
	
	Multivariate distributions with dependent Poisson marginals, which are overlapping sums of
	independent Poisson random variables, like the one on the right-hand side of \eqref{eq:nan-conv},
	are common in the literature (see, e.g., Chapter 37 in \cite{JKB97} and references therein).
	
	\begin{proof}[Proof of Proposition \ref{prop:counts}]
		Let $k \in [N]$ and $\beta \in \{1, \alpha\}$, and define 
		\begin{equation*}
			G^{(N)}_{k, \beta} = \bigl\{(x_1,\ldots,x_N,y):
			0 < x_1 < \ldots < x_k \le \beta < x_{k+1} < \ldots < y \bigr\}.
		\end{equation*}
		Denote by $\supp$ the support of a point measure, that is, the set of its atoms. Since, by construction, $\supp\Xi_n^{(N)}\subset n^{-1}\mathbb{Z}^{N+1}$, we have \[\supp \Xi_n^{(N)}\cap G^{(N)}_{k,\beta}=\supp \Xi_n^{(N)}\cap G^{(N)}_{k,\lfloor\beta n\rfloor/n}.\]
		Indeed, the defining conditions of $G^{(N)}_{k,\beta}$ and $G^{(N)}_{k,\lfloor\beta n\rfloor/n}$ are easily seen to be equivalent on this support.

		By Theorem \ref{TH:MAIN_TH}, combined with the continuous mapping theorem, we have
		\begin{align*}
			C_k(\Pc_{\floor{\beta n}})=\,&\Xi_n^{(N)}\bigl(\bigl\{(x_1,\ldots,x_N,y):
			0 < x_1 < \ldots < x_k \le \tfrac{\floor{\beta n}}{n} < x_{k+1} < \ldots < y \bigr\}\bigr) \\
			=\,&\Xi_n^{(N)}\bigl(G^{(N)}_{k,\lfloor\beta n\rfloor/n}\bigr)
				=\,\Xi_n^{(N)}\bigl(G^{(N)}_{k, \beta}\bigr)
			\xrightarrow[n\to\infty]{d} \, \Xi^{(N)}\bigl(G^{(N)}_{k, \beta}\bigr),
		\end{align*}
		and this convergence holds jointly over all $k$ and $\beta$.
		Note that the sets $G^{(N)}_{k, \beta}$ are unbounded in the Euclidean metric but bounded
		in the sense of localization \eqref{eq:loc}, which justifies the application of the continuous mapping theorem. 
		
		Denote
		\begin{gather*}
			B_{ij}^{(N)} = \bigl\{(x_1,\ldots,x_N,y): 
			0 < x_1 < \ldots < x_i \le 1 < x_{i+1} < \ldots < x_j \le \alpha < x_{j+1} < \ldots < y \bigr\},\\
			B_{i,>N}^{(N)} = \bigl\{(x_1,\ldots,x_N,y): 
			0 < x_1 < \ldots < x_i \le 1 < x_{i+1} < \ldots < x_N < y \le \alpha \bigr\},
		\end{gather*}
		for $1 \le j \le N$, $0 \le i \le j$, and $1 \le j \le N$ respectively. These sets are disjoint, and
		\begin{equation*}
				G_{k,1}^{(N)} = \Bigl(\bigcup_{j=k}^N B_{k,j}^{(N)}\Bigr) \cup B_{k,>N}^{(N)}, \qquad
				G_{k,\alpha}^{(N)} = \bigcup_{i=0}^k B_{i,k}^{(N)}, \qquad k \in [N],
			\end{equation*}
		which yields \eqref{eq:nan-conv} for independent random variables
		\begin{equation*}
			X_{ij} = \Xi^{(N)}\bigl(B_{ij}^{(N)}\bigr) \sim \mathsf{Pois} \bigl(\mu^{(N)}\bigl(B_{ij}^{(N)}\bigr)\bigr),\qquad
			X_{i,>N} = \Xi^{(N)}\bigl(B_{i,>N}^{(N)}\bigr) \sim \mathsf{Pois} \bigl(\mu^{(N)}\bigl(B_{i,>N}^{(N)}\bigr)\bigr).
		\end{equation*}
		The only thing left to show is that
		\begin{gather}
			\label{eq:mul1}
			\mu^{(N)}\bigl(B_{ij}^{(N)}\bigr) = 
			\frac{\theta}{j} \hs \binom{j}{i} \bigl(\alpha^{-1}\bigr)^i \bigl(1-\alpha^{-1}\bigr)^{j-i},\\
			\mu^{(N)}\bigl(B_{i,>N}^{(N)}\bigr) = \frac{\theta}{i}\hs I_{1-\alpha^{-1}}(N-i+1,i).
			\label{eq:mul2}
		\end{gather}
		The equality \eqref{eq:mul1} results from
		\begin{align*}
			\mu^{(N)}\bigl(B_{ij}^{(N)}\bigr)
			=&\int_{0 < x_1 < \ldots < x_i \le 1}\!\!\!\!\dd x_1\ldots\dd x_i \cdot
			\int_{1 < x_{i+1} < \ldots < x_j \le \alpha}\!\!\!\! \dd x_{i+1}\ldots \dd x_j\\\times 
			&\int_{\alpha < x_{j+1} < \ldots < y}\!\!\!\! \theta\,\frac{N!}{y^{N+1}}\,\dd x_{j+1}\ldots\dd y
			=\frac{1}{i!} \cdot \frac{(\alpha-1)^{j-i}}{(j-i)!} \cdot \frac{\theta\hs(j-1)!}{\alpha^j},
		\end{align*}
		which agrees with the right-hand side of \eqref{eq:mul1}. To prove \eqref{eq:mul2},
		the corresponding integral could also be computed explicitly, but it is simpler to note that, by \eqref{eq:ascc}, the equalities
		\begin{equation*}
			\sum_{j=i}^N X_{ij} + X_{i,>N} \deq Z_i,\qquad i\in[N],
		\end{equation*}            
		must hold, where $Z_i \sim \Pois{\frac{\theta}{i}}$. Thus, 
		\begin{equation}\label{eq:suml}
			\E X_{i, >N} = \E Z_i - \sum_{j=i}^N \E X_{ij} = \frac{\theta}{i} - \sum_{j=i}^N \E X_{ij}
		\end{equation}
		where $\E X_{ij}$ is given by \eqref{eq:mul1}, and
		\begin{equation}
			\label{eq:suml-}
			\begin{aligned}
				\sum_{j=i}^N \E X_{ij} &= \theta
				\sum_{j=i}^N \frac{\alpha^{-i}}{j} \binom{j}{i} \bigl(1-\alpha^{-1}\bigr)^{j-i} =
				\theta \sum_{j=i}^N\frac{\alpha^{-i}}{i} \binom{j-1}{i-1} \bigl(1-\alpha^{-1}\bigr)^{j-i} \\ &=
				\frac{\theta}{i} \sum_{j=0}^{N-i}\binom{j+i-1}{j} \alpha^{-i}
				\bigl(1-\alpha^{-1}\bigr)^j = \frac{\theta}{i} \hs\P{X\le N-i},
			\end{aligned}
		\end{equation}
		where $X$ follows the negative binomial distribution with parameters $i$ and $\alpha^{-1}$.
		It is well known (see, e.g., eq.\ (5.31) in \cite{JKK05}) that the cumulative distribution function of $X$
		can be expressed in terms of the normalized incomplete beta function. Hence, by \eqref{eq:suml},
		\begin{equation*}
			\E X_{i, >N} = \frac{\theta}{i} - \frac{\theta}{i} \hs\P{X\le N-i} =
			\frac{\theta}{i} - \frac{\theta}{i} \hs I_{\alpha^{-1}}(i,N-i+1) =
			\frac{\theta}{i} \hs I_{1-\alpha^{-1}}(N-i+1,i),
		\end{equation*}
		which proves \eqref{eq:mul2}.
	\end{proof}
	
	Proposition \ref{prop:counts} can now be rewritten in a more natural infinite-dimensional form.
	
	\begin{corollary}
		For any $\alpha>1$,
		\begin{equation}
			\label{eq:nan-conv2}
			\bigl(C_k(\Pc_n), C_k(\Pc_{\floor{\alpha n}}), k\in\N\bigr) \dto 
			\Bigl(\sum_{j=k}^\infty X_{kj}, \sum_{i=0}^k X_{ik}, k\in\N\Bigr), \qquad n\to\infty,
		\end{equation}
		where $X_{ij}$ are independent and $\Pois{\lambda_{ij}}$-distributed with $\lambda_{ij}$ defined by the first equality in \eqref{eq:lambda}. Here the convergence in distribution is with respect to the product topology in $\R^\infty$.
	\end{corollary}
	\begin{proof}
		It follows from \eqref{eq:suml-} that $\sum_{j=i}^\infty\lambda_{ij}=\frac{\theta}{i}$. Hence, by \eqref{eq:suml},
		\begin{equation}
			\label{eq:ltail}
			\lambda_{i,>N} = \sum_{j=N+1}^\infty\lambda_{ij}.
		\end{equation}
		
		It is well known that to prove convergence in distribution in a space of sequences, it suffices
		to show finite-dimensional convergence. The latter follows from \eqref{eq:nan-conv}
		 and the fact that $X_{i,>N}\deq\sum_{j=N+1}^\infty X_{ij}$, which results from \eqref{eq:ltail}.
	\end{proof}
	
	\subsection{Functional limit theorems}
	We now turn to functional limit theorems in the Skorokhod space and start with block counts.
	By Theorem 6 in \cite{GS23}, the vector-valued processes $\bigl(C_k(\Pc_{\floor{nt}}), k\in [N], t\ge1\bigr)$
	converge in distribution in the $J_1$ topology on $D\bigl([1,+\infty),\R^N\bigr)$ to a certain c\`adl\`ag process
	$\bigl(X_k(t), k \in [N], t\ge1\bigr)$, which can be described as follows.    
	Let $\bigl(Y_k(t), k\in\N, t\ge0 \bigr)$ be a continuous-time homogeneous Markov chain on the space
	$\bigl\{(y_k, k\in\N)\in\mathbb{Z}_+^\infty: \sum_{k=1}^\infty y_k<\infty\bigr\}$ with transition rates
	\begin{equation}
		\label{eq:rates}
		\begin{aligned}
			\theta \quad &\text{for } (y_1,y_2,\ldots) && \to \;\;\; (y_1+1,y_2,\ldots),\\
			ky_{k} \quad &\text{for } (y_1,y_2,\ldots) && \to \;\;\; (y_1,\ldots,y_k-1,y_{k+1}+1,\ldots), \qquad k\ge2,
		\end{aligned}
	\end{equation}
	and independent $Y_k(0)\sim\mathsf{Pois}\bigl(\frac{\theta}{k}\bigr)$. In view of the latter,
	the chain is in steady state; see Theorem 2 in \cite{GS23}. Then $X_k(t)=Y_k(\log t)$, $k\in [N]$, $t\ge1$.
	
	The right-hand side of \eqref{eq:nan-conv2} describes the distribution of $\bigl(X_k(1), X_k(\alpha), k \in \N\bigr)$.
	In a similar way, the entire system of finite-dimensional distributions of the process $(X_k(t),k\in\N, t\ge 1)$ can be derived,
	but the result will be quite involved and not easily tractable. In this sense, Theorem \ref{TH:MAIN_TH},
	which fully describes the block dynamics, is not only more general due to accounting for the composition of blocks
	rather than just their counts, but also encodes this dynamics much more efficiently through the use of random point measures.
	
	Let us now focus on singleton counts. It follows from \eqref{eq:rates} that $Y_1$ is a stationary
	birth-and-death process with birth rate $\lambda_i=\theta$ and death rate $\mu_i=i$, that is,
	it describes the standard $M/M/\infty$ queue in steady state.
	Theorem 6 in \cite{GS23} establishes the convergence of singleton counts to the process $Y_1(\log t)$.
	The following proposition demonstrates how, bypassing the mentioned theorem, one can use Theorem \ref{TH:MAIN_TH}
	to easily prove the convergence of singleton counts and, moreover, constructively describe the limiting process.
	To state it, recall that, for $N=1$, the limiting random measure $\Xi^{(1)}$ in Theorem \ref{TH:MAIN_TH} is
	Poisson on $\{(x,y): 0\le x\le y\}$ with intensity measure $\frac{\theta}{y^2}\,\dd x\,\dd y$.
	
	\begin{proposition}
		\label{prop:fpc}
		Let
		\begin{equation*}
			X_1(t) = \Xi^{(1)}\bigl((0, t]\times(t,+\infty)\bigr), \qquad t>0,
		\end{equation*}
		$r\in\N$, $0 = t_0 < t_1 < \ldots < t_r < t_{r+1} = +\infty$, 
		and $\lambda'_{ij}=\theta(t_i-t_{i-1})(t_j^{-1}-t_{j+1}^{-1})$, $i,j\in[r]$, with $\infty^{-1}=0$ by convention. Then
		\begin{enumerate}[(i)]
			\item the singleton counting processes $\bigl(C_1(\Pc_{\floor{nt}}), t>0\bigr)$ converge as
			$n\to\infty$ to $X_1$ in distribution in the $J_1$ topology on $D\bigl((0, +\infty)\bigr)$,
			\item $\bigl(X_1(t_m),m\in[r]\bigr) \deq \bigl(\sum_{i=1}^m\sum_{j=m}^r X'_{ij}, m\in[r]\bigr)$ with independent
			$X'_{ij}\sim\Pois{\lambda'_{ij}}$,
			\item the finite-dimensional distributions of $X_1$ are defined by their multivariate probability generating function
			\begin{equation*}
				\E\prod_{m=1}^r z_m^{X_1(t_m)} = \exp{\Bigl\{\sum_{1\le i \le j \le r} \lambda'_{ij}(z_iz_{i+1}\ldots z_j-1)\Bigr\}}.
			\end{equation*}
		\end{enumerate}            
	\end{proposition}
	
	\begin{proof}
		Fix an $\eps>0$. The pre-limit and limiting processes $\bigl(C_1(\Pc_{\floor{nt}}), t\ge\eps\bigr)$ and
		$\bigl(X_1(t), t\ge\eps\bigr)$ are c\`adl\`ag and purely jump-type with jumps of size $\pm1$.
		By Lemma 2.12 in \cite{X92}, to establish convergence in $D\bigl([\eps,+\infty)\bigr)$,
		it suffices to show that the jump times and sizes of the pre-limit processes jointly converge
		in distribution to those of the limiting process. Since
		\begin{equation*}
			C_1(\Pc_{\lfloor nt\rfloor})
				=\Xi_n^{(1)}\bigl(\{(x,y):0<x\le\tfrac{\lfloor nt\rfloor}n<y\}\bigr)
				=\Xi_n^{(1)}\bigl(\{(x,y):0<x\le t<y\}\bigr),
		\end{equation*}
		the latter, in turn, follows from Theorem \ref{TH:MAIN_TH} together with the interpretation of vague convergence of measures as convergence of their atoms; see Theorem 3.13 in \cite{R87}.
		As this holds for any $\eps>0$, the convergence in $D\bigl((0,+\infty)\bigr)$ follows by analogy
		with Theorem 16.7 in \cite{B99}, which establishes a similar result for infinity instead of zero.
		
		To prove (ii), note that in the representation 
		\begin{equation}
			\label{eq:X1repr}
			\bigl(X_1(t_m),m\in[r]\bigr) = \bigl(\Xi^{(1)}\bigl((0,t_m]\times(t_m,+\infty)\bigr), m\in[r]\bigr),
		\end{equation}
		the equality
		\begin{equation}
			\label{eq:unun}
			(0,t_m]\times(t_m,+\infty) = \bigcup_{i=1}^m\bigcup_{j=m}^r T_{ij}, \qquad m\in[r],
		\end{equation}
		holds for $T_{ij}=(t_{i-1},t_i]\times(t_j,t_{j+1}]$, $1\le i\le j\le r$, and all $T_{ij}$ are disjoint; see Fig.\ \ref{fig:T}.
		\begin{figure}[t]
			\begin{minipage}{0.5\textwidth}
				\centering
				\begin{tikzpicture}[scale=1.4,>=stealth]
					\fill[gray!30] (0,1.6) -- (1.6,1.6) -- (1.6,4) -- (0,4) -- cycle;
					\draw[thin] (0,1.6) -- (1.6,1.6);
					\draw[thin] (1.6,1.6) -- (1.6,4);
					\draw[very thin,dashed] (1.6,0) -- (1.6,1.6);
					\node[left] at (0,1.6) {\footnotesize $t$};
					\node[below] at (1.6,0) {\footnotesize $t$};
					\draw[->] (0,0) -- (4,0) node[below] {$x$};
					\draw[very thick,->] (0,0) -- (0,4) node[left] {$y$};
					\draw[very thick] (0.01,0.01) -- (4,4);
					\foreach \x/\y in {
						0.089/0.116,
						0.371/1.060,
						0.043/0.225,
						0.053/2.664,
						0.139/0.174,
						1.520/3.840,
						0.270/0.392,
						0.033/0.820,
						0.109/0.155,
						0.291/0.454,
						1.732/2.062,
						0.032/0.187,
						0.084/0.261,
						0.648/1.322,
						0.031/0.134,
						0.495/0.761,
						0.436/0.507,
						0.170/0.377,
						0.609/2.767,
						0.282/1.484,
						0.105/0.145,
						0.222/0.252,
						0.029/0.112,
						0.284/0.546,
						0.115/0.150,
						0.154/1.185,
						0.650/0.707,
						2.428/3.326,
						0.069/0.120,
						0.670/1.718,
						0.031/0.078}
					{\fill[black] (\x,\y) circle (1pt);}
				\end{tikzpicture}
				\caption{The representation of $X_1(t)$.}
			\end{minipage}
			\begin{minipage}{0.5\textwidth}
				\centering
				\begin{tikzpicture}[scale=1.4,>=stealth]
					\fill[gray, opacity=0.3] (0,1) -- (1,1) -- (1,4) -- (0, 4);
					\draw[semithick] (0,1) -- (1,1);
					\draw[semithick] (1,1) -- (1,4);
					\draw[very thin,dashed] (1,0) -- (1,1);
					\node[left] at (0,1) {\footnotesize $t_1$};
					\node[below] at (1,0) {\footnotesize $t_1$};
					\fill[gray, opacity=0.3] (0,1.6) -- (1.6,1.6) -- (1.6,4) -- (0, 4);
					\draw[semithick] (0,1.6) -- (1.6,1.6);
					\draw[semithick] (1.6,1.6) -- (1.6,4);
					\draw[very thin,dashed] (1.6,0) -- (1.6,1.6);
					\node[left] at (0,1.6) {\footnotesize $t_2$};
					\node[below] at (1.6,0) {\footnotesize $t_2$};
					\fill[gray, opacity=0.3] (0,2.5) -- (2.5,2.5) -- (2.5,4) -- (0, 4);
					\draw[semithick] (0,2.5) -- (2.5,2.5);
					\draw[semithick] (2.5,2.5) -- (2.5,4);
					\draw[very thin,dashed] (2.5,0) -- (2.5,2.5);
					\node[left] at (0,2.5) {\footnotesize $t_3$};
					\node[below] at (2.5,0) {\footnotesize $t_3$};
					\fill[gray, opacity=0.3] (0,3) -- (3,3) -- (3,4) -- (0, 4);
					\draw[semithick] (0,3) -- (3,3);
					\draw[semithick] (3,3) -- (3,4);
					\draw[very thin,dashed] (3,0) -- (3,3);
					\node[left] at (0,3) {\footnotesize $t_4$};
					\node[below] at (3,0) {\footnotesize $t_4$};
					\draw[->] (0,0) -- (4,0) node[below] {$x$};
					\draw[very thick,->] (0,0) -- (0,4) node[left] {$y$};
					\draw[very thick] (0.01,0.01) -- (4,4);
					\node at (0.5,1.3) {$T_{11}$};
					\node at (0.5,2.05) {$T_{12}$};
					\node at (0.5,2.75) {$T_{13}$};
					\node at (0.5,3.5) {$T_{14}$};
					\node at (1.3,2.05) {$T_{22}$};
					\node at (1.3,2.75) {$T_{23}$};
					\node at (1.3,3.5) {$T_{24}$};
					\node at (2.05,2.75) {$T_{33}$};
					\node at (2.05,3.5) {$T_{34}$};
					\node at (2.75,3.5) {$T_{44}$};
				\end{tikzpicture}
				\caption{The sets $T_{ij}$.}
				\label{fig:T}
			\end{minipage}
		\end{figure}
		Here, we interpret $(t_r,t_{r+1}]$ as $(t_r,+\infty)$. Denote
		$X'_{ij}=\Xi^{(1)}(T_{ij})$; they are independent and Poisson distributed with means
		\begin{equation}
			\label{eq:l'}
			\E X'_{ij} = \mu^{(1)}(T_{ij}) = \int_{t_{i-1}}^{t_i} \int_{t_j}^{t_{j+1}}\frac{\theta}{y^2}\,\dd x\,\dd y 
			= \theta(t_i-t_{i-1})(t_j^{-1}-t_{j+1}^{-1}) = \lambda'_{ij}.
		\end{equation}
		Hence, (ii) follows from \eqref{eq:X1repr} and \eqref{eq:unun}.
		
		Now, since $\E\hs z^{X'_{ij}}=
		\mathrm e^{\lambda'_{ij}(z-1)}$, we have
		\begin{equation*}
			\E\prod_{m=1}^rz_m^{X_1(t_m)}=\E\prod_{m=1}^r z_m^{\sum_{i=1}^m\sum_{j=m}^rX'_{ij}} 
			= \E\!\prod_{1\le i\le j\le r}\Bigl(\prod_{m=i}^j z_m\Bigr)^{X'_{ij}}
			= \prod_{1\le i\le j\le r}\mathrm e^{\lambda'_{ij}(z_i z_{i+1}\ldots z_j-1)},
		\end{equation*}
		which proves (iii).
	\end{proof}
	
	We now turn to the question of the first singleton in
	$\Pc_{\floor{nt}}$ or, equivalently, the smallest fixed point of $\sigma_{\floor{nt}}$. Let
	\begin{equation*}
		M_n = \min\{l\le n: \{l\} \in \Pc_n\}, \qquad n\in\N,
	\end{equation*}
	and set $M_n=n+1$ if there are no singletons in $\Pc_n$.
	
	For each $t>0$, if there are any atoms of $\Xi^{(1)}$ in $(0,t]\times(t,+\infty)$, let $L(t)$ be
	the $x$-coordinate of the leftmost such atom. Such a leftmost atom exists since $\Xi^{(1)}$ is a.s.~finite on this set. If there are no atoms in that set, define $L(t)=t$. The process $(L(t),t>0)$ is clearly
	non-decreasing and c\`adl\`ag; see Fig.~\ref{fig:L}.
	
	\begin{proposition}~
		\label{prop:ML}
		\begin{enumerate}[(i)]
			\item The processes $\bigl(\frac{M_{\floor{nt}}}{n},t>0\bigr)$ converge as
			$n\to\infty$ to $L$ in distribution in the $J_1$ topology on $D\bigl((0, +\infty)\bigr)$.
			\item The finite-dimensional distributions of $L$ are given as follows. For $r\in\N$,
			let $0<t_1<\ldots<t_r<t_{r+1}=+\infty$ and $0\le x_1\le\dots\le x_r$. Then
		\end{enumerate}
		\begin{equation}
			\label{eq:prop_min_tdf}
			\mathbb{P} \bigl\{L(t_m)>x_m, m\in[r]\bigr\} = 
			\exp\Bigl\{-\theta\sum_{m=1}^r x_m (t_m^{-1}-t_{m+1}^{-1})\Bigr\} \cdot \1\bigl\{x_m<t_m,m\in[r]\bigr\}.
		\end{equation}
	\end{proposition}
	
	Note that, due to the monotonicity of $L$, it indeed suffices to specify the probabilities
	on the left-hand side of \eqref{eq:prop_min_tdf} only for non-decreasing $x_m$.
	
	\begin{figure}[t]
		\begin{minipage}{0.5\textwidth}
			\centering
			\begin{tikzpicture}[scale=1.7,>=stealth]
				\draw[->] (0,0) -- (3,0) node[below] {$x$};
				\draw[very thick,->] (0,0) -- (0,3) node[left] {$y$};
				\draw[very thick] (0.01,0.01) -- (3,3);
				\foreach \x/\y in {
					0.023/0.046,
					0.047/0.072,
					0.089/0.116,
					0.371/1.06,
					0.043/0.225,
					0.139/0.174,
					0.27/0.392,
					0.033/0.82,
					0.109/0.125,
					0.291/0.454,
					2.132/2.767,
					0.032/0.187,
					0.084/0.261,
					0.648/1.322,
					0.031/0.134,
					0.495/0.761,
					0.436/0.507,
					0.17/0.377,
					0.282/1.484,
					0.115/0.145,
					0.222/0.252,
					0.029/0.112,
					0.284/0.546,
					0.135/0.15,
					0.154/1.185,
					0.65/0.707,
					0.079/0.12,
					0.67/1.718}
				{\fill[black] (\x,\y) circle (1pt);}
			\end{tikzpicture}
		\end{minipage}
		\begin{minipage}{0.5\textwidth}
			\centering
			\begin{tikzpicture}[scale=1.7,>=stealth,
				spy/.style={%
					draw,gray!75,line width=0.75pt,circle,inner sep=0pt,}]
				\def\spyviewersize{2cm}
				\def\spyonclipreduce{0.7pt}
				\def\pik{
					\draw[->] (0,0) -- (3,0) node[below] {$t$};
					\draw[->] (0,0) -- (0,3) node[left] {$L(t)$};
					\draw[dash pattern=on 0.3pt off 0.3pt,semithick] (0,0) -- (0.01,0.01);
					\foreach \x/\y/\z in {
						0.010/0.010/0.029,
						0.029/0.029/0.112,
						0.112/0.031/0.134,
						0.134/0.032/0.187}
					{\draw[thick] (\x,\y) -- (\z-0.005,\y);
						\draw[densely dashed] (\x,\y) -- (\x,0);};
					\foreach \x/\y/\z in {
						0.187/0.033/0.820,
						0.820/0.154/1.185,
						1.185/0.282/1.484,
						1.484/0.670/1.718}
					{\draw[thick, arrows = {-Latex[width=0pt 3, length=5pt]}] (\x,\y) -- (\z,\y);
						\draw[densely dashed] (\x,\y) -- (\x,0);};
					
					\draw[thick, arrows = {-Latex[width=0pt 4, length=5pt]}] (2.132,2.132) -- (2.767,2.132);
					\draw[thick] (1.718,1.718) -- (2.132,2.132);
					\draw[densely dashed] (1.718,1.718) -- (1.718,0);
					\draw[thick] (2.767,2.767) -- (3,3);
					\draw[densely dashed] (2.767,2.767) -- (2.767,0);
				}
				\pik
				\def\spyfactorI{11}
				\coordinate (spy-on 1) at (0.15,0.03);
				\coordinate (spy-in 1) at (0.9,2);
				\begin{scope}
					\clip (spy-in 1) circle (0.3*\spyviewersize-\spyonclipreduce);
					\pgfmathsetmacro\sI{1/\spyfactorI}
					\begin{scope}[shift={($\sI*(spy-in 1)-\sI*(spy-on 1)$)},scale around={\spyfactorI:(spy-on 1)}]
						\pik
						\foreach \x/\y/\z in {
							0.029/0.029/0.112,
							0.112/0.031/0.134,
							0.134/0.032/0.187}
						{\draw[thick, arrows = {-Latex[width=0pt 3, length=4pt]}] (\x,\y) -- (\z,\y);}
					\end{scope}
				\end{scope}
				\node[spy,minimum size={\spyviewersize/\spyfactorI}] (spy-on node 1) at (spy-on 1) {};
				\node[spy,minimum size=\spyviewersize] (spy-in node 1) at (spy-in 1) {};
				\draw [spy] (spy-on node 1) -- (spy-in node 1);
			\end{tikzpicture}
		\end{minipage}
		\caption{A typical sample of atoms of $\Xi^{(1)}$ with the corresponding sample path of $L$.}
		\label{fig:L}
	\end{figure}
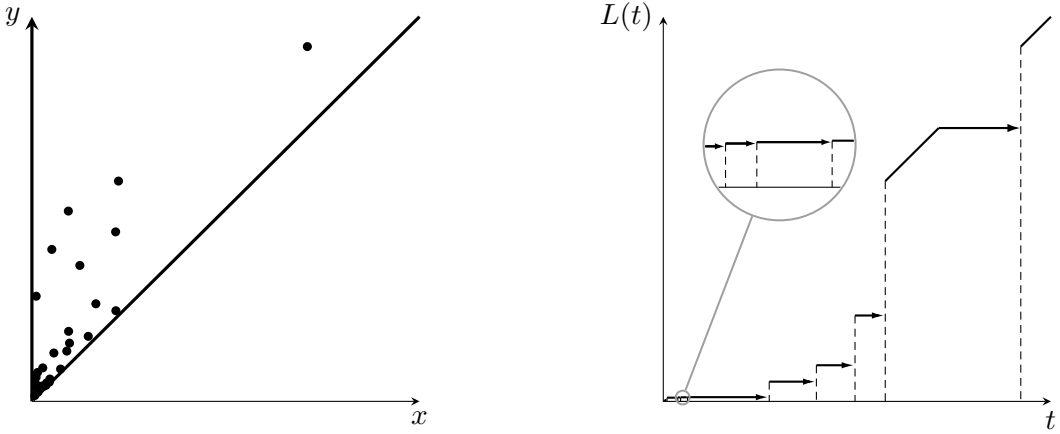
	
	\begin{proof}[Proof of Proposition \ref{prop:ML}]
		As in the proof of Proposition~\ref{prop:fpc}, it suffices to establish convergence in $D\bigl([\eps,+\infty)\bigr)$ for a fixed $\eps>0$. Since the limiting process $L$ is no longer purely jump-type, we cannot directly apply Lemma 2.12 in \cite{X92} as before and therefore proceed with a somewhat refined argument. It is worth mentioning that convergence in the weaker $M_1$ topology follows immediately from the finite-dimensional convergence, implied by Theorem~\ref{TH:MAIN_TH}, and monotonicity of the pre-limit processes (see, e.g., Corollary~12.5.1 in~\cite{W02}).
			
		Define $M'_{\floor{nt}}$ and $M''_{\floor{nt}}$ by the same rule as $M_{\floor{nt}}$, 
		except that $M'_{\floor{nt}}=0$ and $M''_{\floor{nt}}=nt$ whenever $\Pc_{\floor{nt}}$ contains no singletons. 
		Similarly, define $L'(t)$ as $L(t)$ but with value $0$ instead of $t$ if there are no atoms of $\Xi^{(1)}$ in 
		$(0,t]\times(t,+\infty)$. Unlike $L$, $L'$ is purely jump-type but fails to be non-decreasing.
		The convergence of $\Bigl(\frac{M'_{\floor{nt}}}{n}, t\ge\eps\Bigr)$ to $L'$ in $D\bigl([\eps,+\infty)\bigr)$ 
		follows from the same argument as in the proof of Proposition \ref{prop:fpc}. By Skorokhod’s representation theorem, we may and do place the processes on a common probability space so that this convergence holds a.s.
						
		Fix a $\delta>0$. Consider the set
		\[\Cc=\bigl\{x\in D\bigl([\eps,+\infty)\bigr)\colon x(t)\in\{0\}\cup(\delta,+\infty)\text{ for all }t\ge\eps\bigr\}\]
		and the mapping from $\Cc$ to $D\bigl([\eps,+\infty)\bigr)$ given by
		\begin{equation}
		\label{eq:map}
		\bigl(g(t),t\ge\eps\bigr)\mapsto
		\bigl(g(t) \hs \1\{g(t)\ne0\} + t\hs\1\{g(t)=0\}, t\ge\eps\bigr).
		\end{equation}
		Hence, $\frac{M'_{\floor{n\cdot}}}{n} \mapsto \frac{M''_{\floor{n\cdot}}}n$ and $L'\mapsto L$
		for all $\omega$ such that $L'$, and thus $\frac{M'_{\floor{n\cdot}}}{n}$ eventually, lie in $\Cc$.
		Since, by construction, this necessary holds for some $\delta(\omega)>0$, and the map~\eqref{eq:map} is $J_1$-continuous, it follows that $\tfrac{M''_{\floor{n\cdot}}}n$ converges to $L$ in $D\bigl([\eps,+\infty)\bigr)$.			
		As
		\begin{equation*}
			\frac{M''_{\floor{nt}}}{n} \le
			\frac{M^{\vphantom a}_{\floor{nt}}}{n} \le
			\frac{M''_{\floor{nt}}}{n} +\frac{1}{n},
		\end{equation*}
		(i) follows.
		
		To prove (ii), first note that
		$\mathbb{P}\bigl\{L(t_m)>x_m\bigr\} = 0$ for
		$x_m\ge t_m$ due to $L(t_m)\le t_m$. For $x_m < t_m$, denote $U_m = (0,x_m]\times(t_m, +\infty)$ and observe that
		\begin{align*}
			\mathbb{P}\Bigl\{L(t_m) > x_m, m\in[r]\Bigr\} = \mathbb{P}\Bigl\{\Xi^{(1)}\Bigl(\bigcup_{m=1}^r U_m\Bigr) = 0\Bigr\}
			= \exp\Bigl\{-\mu^{(1)}\Bigl(\bigcup_{m=1}^rU_m\Bigr)\Bigr\}.
		\end{align*}
		A calculation similar to \eqref{eq:l'} shows that the right-hand side here is the same as in \eqref{eq:prop_min_tdf}.
	\end{proof}
	
	\subsection{Short-lived singletons}

	We now examine singletons with a short lifetime. First of all, we must exclude from consideration singletons
	born in the early stages of CRP formation, as in the initial \lq inflationary expansion\rq\ phase of the process,
	their birth and growth proceed at a singularly fast rate. From a mathematical perspective, this can be explained
	by the scale invariance of the measure $\Xi^{(1)}$: its structure on short time intervals near zero is just
	as irregular as on long ones in later stages; see Remark \ref{rem:si}. Formally, for any $0 < \delta_0 < \delta$,
	the number of singletons born between $\floor{\delta_0 n}+1$ and $\floor{\delta n}$ is
	$\Xi^{(1)}_n\bigl(\{(x,y): \delta_0 < x \le \delta\}\bigr)$,
	thus converging to a Poisson random variable with mean $\log\frac{\delta}{\delta_0}$.
	
	For $n\in\N$ and $\delta>0$, denote by $S_{\delta,n}$ the birth time of the singleton born at or
	after $\floor{\delta n}$ that transformed into a doubleton in the shortest time among all such singletons,
	and by $T_{\delta,n}$ its lifetime.
	We now describe the joint asymptotics of $\bigl(S_{\delta,n},T_{\delta,n}\bigr)$ as $n\to\infty$.        
	
	\begin{proposition}
		\label{prop:ST}
		Let $\bigl(S_{\delta},T_{\delta}\bigr)$ be a random vector with density
		\begin{equation}
			\label{eq:ST}
			f(s,t) = \frac{\theta}{(s+t)^2}\Bigl(1+\frac{t}{\delta}\Bigr)^{-\theta}, \qquad s\ge\delta,\,t\ge0.
		\end{equation}
		Then
		\begin{equation}
			\label{eq:STconv}
			\biggl(\frac{S_{\delta,n}}{n},\frac{T_{\delta,n}}{n}\biggr) \dto
			\bigl(S_{\delta},T_{\delta}\bigr), \qquad n\to\infty.
		\end{equation}
	\end{proposition}
	
	Equation \eqref{eq:ST} implies that $T_\delta$ follows a Pareto-type distribution with density 
	$\frac{\theta}{\delta}\bigl(1+\frac{t}{\delta}\bigr)^{-\theta-1}$, $t\ge0$,
	while the marginal density of $S_\delta$ is more intricate and can be expressed in terms of hypergeometric functions.
	
	\begin{proof}[Proof of Proposition \ref{prop:ST}]
		Letting 
		$\Delta = \bigl\{(x,y)\in\X_1: x \ge\delta\bigr\}$, define $T_\delta$ as $\min\,(y-x)$ over 
		$(x,y)\in\Delta\cap\supp\Xi^{(1)}$, and $S_\delta$ as the $x$-coordinate of the corresponding $\argmin$. 
		By similar arguments to those in the previous propositions, \eqref{eq:STconv} follows from Theorem \ref{TH:MAIN_TH}.    
			
		We will now prove \eqref{eq:ST}. For any Borel function $g:[\delta,+\infty)\times[0,+\infty)\to[0,+\infty)$, we have
		\begin{align*}
			&g(S_\delta,T_\delta) \\&=\hspace{-10pt}\sum_{(x,y)\in\Delta\cap\supp\Xi^{(1)}}\hspace{-10pt}
			g(x,y-x)\,\mathds{1}\bigl\{
			y-x\le y'-x'\;\;\forall(x',y') \in\Delta\cap\supp\Xi^{(1)}
			\bigr\}\\
			&= \int_\Delta g(x,y-x)\,\mathds{1}\bigl\{
			y-x\le y'-x'\;\;\forall(x',y') \in\Delta\cap\supp\Xi^{(1)}
			\bigr\}\,\Xi^{(1)}(\dd x,\dd y).
		\end{align*}
		Note that, in the above sum, a.s.~only one summand is nonzero. 
		Now, by the Mecke equation (see, e.g., Theorem 4.1 in \cite{LP18}),
		\begin{equation}
			\begin{aligned}
				\label{eq:EgST}
				\E g(S_\delta,T_\delta)&=\int_\Delta g(x,y-x)\,\\&\times\mathbb P
				\bigl\{y-x\le y'-x'\;\forall (x',y')\in\Delta\cap\supp\Xi^{(1)}
				\bigr\}\,\mu^{(1)}(\dd x,\dd y).
			\end{aligned}
		\end{equation}
		Let $B = \{(x',y')\in\X_1: x'\ge\delta, y'<x'+y-x\}$. Then
		the probability under the integral sign equals $\mathbb{P}\{\Xi^{(1)}(B)=0\}=\exp\{-\mu^{(1)}(B)\}$, where
		\begin{equation*}
			\mu^{(1)}(B) = \int_{\delta}^{+\infty}\dd x'\int_{x'}^{x'+y-x}\!\!\frac{\theta}{(y')^2}\,\dd y'
			= \theta\int_\delta^{+\infty}\biggl(\frac{1}{x'} - \frac{1}{x'+y-x}\biggr)\,\dd x'
			= \theta \log\Bigl(1+\frac{y-x}{\delta}\Bigr).
		\end{equation*}
		Thus, by \eqref{eq:EgST}, we have
		\begin{equation*}
			\E g(S_\delta,T_\delta) = \int_{\Delta} g(x,y-x)\cdot\frac{\theta}{y^2} \Bigl(1+\frac{y-x}{\delta}\Bigr)^{-\theta}\dd x\,\dd y
			= \iint\limits_{\substack{s\ge\delta,\\t\ge 0}} g(s,t) \cdot \frac{\theta}{(s+t)^2}\Bigl(1+\frac{t}{\delta}\Bigr)^{-\theta} \dd s\,\dd t,
		\end{equation*}\\[-20pt]
		which yields \eqref{eq:ST}.    
	\end{proof}     
	
	We conclude with a functional limit theorem for the number of short-lived singletons.
	For $n\in\N$, $\delta>0$, and $t \ge 0$, denote by $Q_{\delta,n}(t)$ the number of singletons
	born at or after $\floor{\delta n}$ that transformed into doubletons within time  $\floor{n t}$ after their birth.
	
	\begin{proposition}
		\label{prop:Z}
		Let $(Z(t), t\ge0)$ be a unit-rate Poisson counting process. Then the processes $Q_{\delta,n}(t)$
		converge as $n\to\infty$ to $\bigl(Z\bigl(\theta\log\bigl(1+\frac{t}{\delta}\bigr)\bigr),t\ge0\bigr)$
		in distribution in the $J_1$ topology on $D\bigl([0,+\infty)\bigr)$.
	\end{proposition}
	
	\begin{proof}
		As in the proof of Proposition \ref{prop:fpc}, 
		\begin{equation*}
			\bigl(Q_{\delta,n}(t),t\ge0\bigr)\xrightarrow{d}
			\bigl(\Xi^{(1)}\bigl(\bigl\{(x,y): x\ge\delta,x\le y \le x+t\bigr\}\bigr),t\ge0\bigr), \quad n\to\infty,
		\end{equation*}
		on $D\bigl([0,+\infty)\bigr)$.
		Since $\Xi^{(1)}$ is a Poisson measure, the process on the right-hand side has independent Poisson increments.
		The increment between times $t_1$ and $t_2$ has mean
		\begin{equation*}
			\int_\delta^{+\infty}\!\dd x\int_{x+t_1}^{x+t_2} \frac{\theta}{y^2}\,\dd y
			= \theta\log\Bigl(1+\frac{t_2}{\delta}\Bigr) - \theta\log\Bigl(1+\frac{t_1}{\delta}\Bigr),
		\end{equation*}
		which proves the claim.
	\end{proof}
	
	Similarly, one can obtain counterparts of Propositions \ref{prop:ST} and \ref{prop:Z} for blocks
	that rapidly grow from size $1$ to $N$ or even from $N_1$ to $N_2$. However, the limiting distributions
	and processes turn out to be less explicit, as they are expressed in terms of involved special functions.

	\vspace{10pt}
	
	\noindent
	\textsc{Igor Sikorsky Kyiv Polytechnic Institute, Prospect Beresteiskyi 37, 03056, Kyiv, Ukraine} \\
	\textit{Email address}: \texttt{galganov.oleksii@lll.kpi.ua}
	
	\vspace{10pt}
	
	\noindent
	\textsc{Igor Sikorsky Kyiv Polytechnic Institute, Prospect Beresteiskyi 37, 03056, Kyiv, Ukraine;} \\
	\textsc{Institute of Mathematical Statistics and Actuarial Science, University of Bern, Alpeneggstrasse 22, CH-3012, Bern, Switzerland} \\
	\textit{Email address}: \texttt{andrii.ilienko@unibe.ch}
		

\begin{thebibliography}{10}
		\bibitem{P06}
		J.~Pitman.
		\newblock {\em Combinatorial stochastic processes. {Ecole} d'{Et{\'e}} de
			{Probabilit{\'e}s} de {Saint}-{Flour} {XXXII} -- 2002.}, volume 1875 of {\em
			Lect. Notes Math.}
		\newblock Berlin: Springer, 2006.

		\bibitem{BGR01}
		M.~J. Beal, Z.~Ghahramani, and C.~E. Rasmussen.
		\newblock The infinite hidden Markov model.
		\newblock In {\em Proceedings of the 15th International Conference on Neural
			Information Processing Systems: Natural and Synthetic}, NIPS'01, pages
		577--584. MIT Press, 2001.

		\bibitem{TJBB06}
		Y.~W. Teh, M.~I. Jordan, M.~J. Beal, and D.~M. Blei.
		\newblock Hierarchical {Dirichlet} processes.
		\newblock {\em J. Am. Stat. Assoc.}, 101(476):1566--1581, 2006.

		\bibitem{FSJW11}
		E.~B. Fox, E.~B. Sudderth, M.~I. Jordan, and A.~S. Willsky.
		\newblock A sticky {HDP}-{HMM} with application to speaker diarization.
		\newblock {\em Ann. Appl. Stat.}, 5(2A):1020--1056, 2011.

		\bibitem{S24}
		D.~Stark.
		\newblock Markov chains generating random permutations and set partitions.
		\newblock {\em Stochastic Processes Appl.}, 178:18, 2024.

		\bibitem{Cr16a}
		H.~Crane.
		\newblock The ubiquitous {Ewens} sampling formula.
		\newblock {\em Stat. Sci.}, 31(1):1--19, 2016.

		\bibitem{Cr16b}
		H.~Crane.
		\newblock Rejoinder: {The} ubiquitous {Ewens} sampling formula.
		\newblock {\em Stat. Sci.}, 31(1):37--39, 2016.
		
		\bibitem{ABT03}
		R.~Arratia, A.~D. Barbour, and S.~Tavar{\'e}.
		\newblock {\em Logarithmic combinatorial structures: {A} probabilistic
			approach}.
		\newblock EMS Monogr. Math. Z{\"u}rich: European Mathematical Society (EMS),
		2003.
		
		\bibitem{DJ91}
		P.~Donnelly and P.~Joyce.
		\newblock Consistent ordered sampling distributions: characterization and
		convergence.
		\newblock {\em Adv. in Appl. Probab.}, 23(2):229--258, 1991.

		\bibitem{G04}
		A.~V. Gnedin.
		\newblock Three sampling formulas.
		\newblock {\em Combin. Probab. Comput.}, 13(2):185--193, 2004.

		\bibitem{GIM12}
		A.~Gnedin, A.~Iksanov, and A.~Marynych.
		\newblock A generalization of the {E}rd{\H{o}}s-{T}ur\'{a}n law for the order
		of random permutation.
		\newblock {\em Combin. Probab. Comput.}, 21(5):715--733, 2012.

		\bibitem{DGM24}
		Z.~Derbazi, A.~Gnedin, A.~Marynych.
		\newblock Records in the infinite occupancy scheme,
		\newblock{\em ALEA, Lat. Am. J. Probab. Math. Stat.}, 21(2):1475--1493, 2024.

		\bibitem{I19}
		A.~Ilienko.
		\newblock Convergence of point processes associated with coupon collector's
			and {Dixie} cup problems, 
		\newblock{\em Electron. Commun. Probab.}, 24:9, id/No 51, 2019.

		\bibitem{GS23}
		A.~Gnedin and D.~Stark.
		\newblock Random permutations and queues.
		\newblock {\em Adv. in Appl. Math.}, 149:Paper No. 102549, 26, 2023.

		\bibitem{GW24}
		J.~Garza and Y.~Wang.
		\newblock Limit theorems for random permutations induced by {C}hinese
		restaurant processes, 2024.
		\newblock Preprint, available at \url{https://arxiv.org/abs/2412.02162}.

		\bibitem{BMMU24}
		J.~E. Bj{\"o}rnberg, C.~Mailler, P.~M{\"o}rters, D.~Ueltschi.
		\newblock A two-table theorem for a disordered {Chinese} restaurant process, 
		\newblock{\em Ann. Appl. Probab.}, 34(6):5809--5841, 2024.

		\bibitem{GI24}
		O.~Galganov and A.~Ilienko.
		\newblock Short cycles of random permutations with cycle weights: point
		processes approach.
		\newblock {\em Statist. Probab. Lett.}, 213:Paper No. 110169, 7, 2024.

		\bibitem{K17}
		O.~Kallenberg.
		\newblock {\em Random measures, theory and applications}, volume~77 of {\em
			Probab. Theory Stoch. Model.}
		\newblock Cham: Springer, 2017.
		
		\bibitem{R87}
		S.~I. Resnick.
		\newblock {\em Extreme values, regular variation, and point processes},
		volume~4 of {\em Appl. Probab.}
		\newblock Springer-Verlag, New York, NY, 1987.

		\bibitem{JKB97}
		N.~L. Johnson, S.~Kotz, and N.~Balakrishnan.
		\newblock {\em Discrete multivariate distributions}.
		\newblock Wiley Series in Probability and Statistics: Applied Probability and
		Statistics. John Wiley \& Sons, Inc., New York, 1997.

		\bibitem{JKK05}
		N.~L. Johnson, A.~W. Kemp, and S.~Kotz.
		\newblock {\em Univariate discrete distributions}.
		\newblock Wiley Series in Probability and Statistics. Wiley-Interscience [John
		Wiley \& Sons], Hoboken, NJ, third edition, 2005.

		\bibitem{X92}
		A.~Xia.
		\newblock Weak convergence of jump processes.
		\newblock In {\em S\'eminaire de probabilit\'es XXVI}, pages 32--46. Berlin:
		Springer-Verlag, 1992.
		
		\bibitem{B99}
		P.~Billingsley.
		\newblock {\em Convergence of probability measures.}
		\newblock Wiley Ser. Probab. Stat. Chichester: Wiley, 2nd ed. edition, 1999.

		\bibitem{W02}
		W.~Whitt.
		\newblock{\em Stochastic-process limits. {An} introduction to stochastic-process
			limits and their application to queues.}
		\newblock Springer Ser. Oper. Res., New York, NY: Springer, 2002.

		\bibitem{LP18}
		G.~Last and M.~Penrose.
		\newblock {\em Lectures on the {P}oisson process}, volume~7 of {\em Institute
			of Mathematical Statistics Textbooks}.
		\newblock Cambridge University Press, Cambridge, 2018.
		
	\end{thebibliography}
\end{document}